\documentclass[a4, 10 pt, conference]{ieeeconf}
\IEEEoverridecommandlockouts 
\usepackage{amsmath,amssymb,amsfonts}

\usepackage[T1]{fontenc}
\usepackage[english]{babel}
\usepackage[latin9]{inputenc}
\usepackage{graphicx}
\usepackage{subcaption}
\usepackage{comment}
\usepackage{url}
\usepackage[font=footnotesize]{caption}
\usepackage{bbding}
\usepackage{color,xcolor}

\usepackage{enumerate}
\usepackage{times}
\usepackage[cal=boondox]{mathalfa}

\usepackage{dutchcal}

\usepackage{breqn}

\usepackage[mathscr]{euscript}

\newtheorem{theorem}{Theorem}
\newtheorem{prop}{Proposition}

\providecommand{\norm}[1]{{\left\lVert#1\right\rVert}}

\providecommand{\pr}[1]{{\left(#1\right)}} 
\providecommand{\pp}[1]{{\left[#1\right]}}
\providecommand{\set}[1]{{\left\lbrace#1\right\rbrace}}

\newcommand{\normi}[3]{\norm{#1}_{
		\ifthenelse{\equal{#2}{1}}{H_0^1\pr{\mathcal{O}_{#3}}}{%
			\ifthenelse{\equal{#2}{-1}}{H^{-1}\pr{\mathcal{O}_{#3}}}{}}}}

\begin{document}

\title{\LARGE \bf Stochastic Control of Addiction with State-Dependent Jump Relapse}

\author{\sc
Dounia Aissi,
Ioana Ciotir,
Dan Goreac,
and Juan Li%
\thanks{D. Aissi is a MSc student at Universit\'{e} Laval, 2425, rue de l'Agriculture, Qu\'{e}bec (QC) G1V 0A6, Canada
        (e-mail: dounia.aissi@crchudequebec.ulaval.ca).}%
\thanks{I. Ciotir is with INSA de Rouen, Normandie University, LMI (EA 3226 -- FR CNRS 3335), 76000 Rouen, France
        (e-mail: ioana.ciotir@insa-rouen.fr).}%
\thanks{D. Goreac is with \'Ecole d'Actuariat, Universit\'{e} Laval, 2425, rue de l'Agriculture, Qu\'{e}bec (QC) G1V 0A6, Canada
        (e-mail: dan.goreac@act.ulaval.ca).}%
\thanks{D. Goreac and J. Li are with the School of Mathematics and Statistics, Shandong University Weihai, Weihai 264209, P.R. China
        (e-mail: juanli@sdu.edu.cn).}%
\thanks{D. Goreac is also with Universit\'{e} Gustave Eiffel, LAMA (UMR 8050), UPEM, UPEC, CNRS, F-77454 Marne-la-Vall\'{e}e, France
        (e-mail: dan.goreac@univ-eiffel.fr).}%
}

\maketitle

\begin{abstract}
We study a continuous-time rational addiction model where addiction capital follows a piecewise deterministic Markov process with state-dependent jumps capturing relapse and recovery. The instantaneous utility combines consumption and addiction capital via a power specification, leading to Hamilton-Jacobi-Bellman (HJB) equations with nonlocal jump terms. In the capped, bounded-control case we obtain a unique bounded viscosity solution and prove that optimal policies are bang-bang, switching between minimal and maximal consumption, while in the uncapped case we derive explicit linear feedback controls and closed-form value functions in several parameter regimes.


\end{abstract}
\section{Introduction}
\subsection{Historical notes on the Becker-Murphy model for addiction}
In the deterministic rational-addiction framework of Becker and Murphy \cite{BeckerMurphy1988}, the representative agent maximizes an intertemporal utility functional
\[
\sum_{t=0}^{\infty} \beta^{t} \, u\bigl(c_t, S_t\bigr),
\qquad 0<\beta<1,
\]
where $c_t$ denotes current consumption of the addictive good and $S_t$ is an ``addiction stock'' summarizing past consumption. The evolution of the stock is typically given by a linear law of motion
\[
S_{t+1} = (1-\delta)\,S_t + c_t,
\qquad \delta \in (0,1),
\]
so that past consumption raises $S_t$ and future marginal utility of $c_t$ (reinforcement), while the depreciation rate $\delta$ captures forgetting and recovery. Preferences are specified so that (i) higher $S_t$ increases the marginal utility of $c_t$ (reinforcement), (ii) increases in $S_t$ shift future utility down if $c_t$ is not maintained (withdrawal), and (iii) forward-looking agents internalize how current $c_t$ changes future $S_{t+k}$ and thus future utility, see \cite{BeckerMurphy1988,GruberKoszegi2001}. This leads to Euler equations linking $c_t$ and $c_{t+1}$ and to testable predictions such as adjacent complementarity (consumption today and tomorrow move together) and long-run price elasticities that may exceed short-run ones.

The stochastic Becker--Murphy models, e.g., in \cite{YangZhang2022}, extend this structure by adding a Brownian-driven noise to the law of motion of the addiction stock, so that $S_t$ (or $A(t)$ in continuous time) becomes a controlled diffusion; this preserves the forward-looking optimality of the original model while generating realistic features such as random relapses, endogenous cycles of abstinence and binge, and explicit conditions separating explosive, stable, or declining addiction regimes, all within a rational, utility-maximizing framework.
\subsection{Historical notes on PDMPs and control}

Piecewise-deterministic Markov processes (PDMPs) were formally introduced by Davis in the early 1980s as a general class of non-diffusion stochastic models combining deterministic flows with random jumps, see \cite{Davis1984}. The theory was developed further in his monograph on Markov models and stochastic optimization \cite{Davis1993}, where PDMPs appear as a unifying framework for many applied probability models.

The optimal control of such processes was first systematically addressed by Vermes, who introduced the notion of piecewise open-loop controls and studied associated optimality conditions in continuous time in \cite{Vermes1986}. In parallel, Soner considered optimal control problems with state-space constraints for piecewise deterministic processes and laid some of the viscosity-solution foundations for constrained HJB equations in this setting, cf. \cite{Soner1986}. Building on these works, Costa and Dufour and co-authors developed a comprehensive theory of PDMP control, including discounted and average-cost criteria, as well as mixed gradual/impulsive control, with characterizations via dynamic programming and vanishing discount techniques in \cite{CostaDufour2008,Dufour2016}.

More recent contribution concern a viability and invariance theory for controlled PDMPs, using viscosity solutions and normal-cone conditions to characterize viability kernels and reachability sets in applications to gene networks in \cite{GoreacViab2012}. In addition, the author has studied asymptotic control and linear-programming-type formulations for discounted PDMP problems, with occupation measure techniques in \cite{GoreacSICON2015} devoted to models inspired by temperate viruses. In the same spirit, \cite{GoreacSereaZubov2012} develops a capture basin method to treat controlled PDMPs.
\subsection{Towards a better formulation of noise in the Becker-Murphy model}
A PDMP framework is often more appropriate than a pure Brownian model for addiction dynamics because it separates deterministic evolution from random jump events. Between jumps, the addiction state follows an ODE related to tolerance, slow recovery, and so on. Relapses or treatment entries occur at random times as state-dependent jumps, with intensity $\lambda(x)$ directly interpretable as a relapse hazard, see, for instance, \cite{Davis1984,Campillo2019}. This matches clinical descriptions of addiction as a chronic condition with relatively stable phases punctuated by sharp episodes of heavy use or crisis, which are not well represented by small continuous Brownian fluctuations alone, cf. \cite{AddictionNeuroModel2022}. Moreover, PDMPs align naturally with hybrid control problems (continuous dosage vs. impulsive interventions) and event-time data, since $\lambda(x)$ can be estimated and interpreted like a survival-model hazard, whereas Brownian volatility has no such direct behavioural meaning.
\subsection{Structure and findings}
Section~\ref{Sec2} introduces the addiction-capital dynamics, the jump mechanism, and the economic meaning of the model parameters, and establishes basic consistency properties such as existence, nonnegativity, and comparison of addiction paths with respect to consumption levels. Sections~\ref{Sec3} and~\ref{Sec4} analyze, respectively, the capped (bounded) and uncapped consumption frameworks. In the capped case, we characterize the value function as the unique bounded viscosity solution of the associated HJB equation (Theorem~\ref{Th1}) and show that optimal controls are of bang-bang type (Theorem~\ref{Th2}). In the uncapped case, we obtain explicit feedback formulas for the optimal consumption rule and the value function for a broad class of parameters (Theorem~\ref{Th3}). The implications of these results for addiction behavior are discussed in Section~\ref{Sec5}, while conclusions and avenues for future research are presented in Section~\ref{Sec6}; technical proofs of auxiliary results are collected in Section~\ref{SecApp}.
\section{Extended Stochastic Becker--Murphy Model with Classical PDMP Noise}\label{Sec2}

\subsection{Setup and objective}

We retain the Becker--Murphy rational addiction framework, altering the addiction capital as a piecewise deterministic Markov process (PDMP) with state-dependent jump intensity. The addict chooses a consumption path \(c(t)\) to maximise
\begin{equation}\label{Cost}
  \max_{c(\cdot)} \;
  \mathbb{E}\left[\int_0^{+\infty} u(c(t),A(t))\,e^{-\rho t}\,dt\right],
\end{equation}
where \(A(t)\) is the addiction capital and \(\rho>0\) is the time preference parameter.

The instantaneous utility is taken to be proportional to a power of the (upper-capped) addiction capital 
\begin{equation}
  u(c,A) = - \min\set{A,A_{\max}}^{\omega} c^{-\phi},
\end{equation}
with \(\phi>0\) and \(\omega>0\).

\subsection{PDMP dynamics for addiction capital}

The addiction capital \(A(t)\) evolves as a PDMP:
\begin{equation}\begin{split}
  dA(t) = &\bigl[c(t) - \delta A(t)\bigr]\,dt \\&+ \pr{\theta_1+\theta_2\min\set{ A(t^-),A_{\max}}}\,dN(t),
  \end{split}
  \label{eq:PDMP-classic-A}
\end{equation}
where
\begin{itemize}
  \item between jumps, \(A(t)\) satisfies the ODE \(\dot A(t) = c(t) - \delta A(t)\);
  \item $A_{\max}$ is a realistic maximal cap; it will be used to bound both the jumping mechanism (intensity and incremental jump) and the utility. We will explain both the case when \begin{enumerate}
      \item $A_{\max}=\infty$, but jumps are only linked to increase addiction, or
      \item $A_{\max}<\infty$, but, in this case, curative measures can translate in a decrease in addiction, allowing more freedom;
  \end{enumerate}
  \item at each jump time \(\tau_k\), \[A(\tau_k) = A(\tau_k^-)+\theta_1+\theta_2\min\set{A(\tau_k^-),A_{\max}},\] with fixed parameters \(\theta_j\in\mathbb{R}\);
  \item \(N(t)\) is a counting process with state-dependent intensity \(\lambda(A(t))\).
\end{itemize}

A classical PDMP specification assumes
\begin{equation}
  \lambda(a) = \lambda_1 + \lambda_2 \min\set{a,A_{\max}}, \qquad a \ge 0,
\end{equation}
with constants \(\lambda_1 \geq 0\), \(\lambda_2 \ge 0\). Thus the jump rate increases linearly with the current addiction level \(A(t)\), and jumps occur according to the usual Markov jump mechanism defined by \(\lambda(\cdot)\) and the post-jump map \(a \mapsto a+\theta_1+\theta_2\min\set{a,A_{\max}}\). Fur further reference, we denote by
\begin{equation}\label{J}J(a):=\theta_1+\theta_2\min\set{a,A_{\max}},\ \text{for }a\in\mathbb{R}.\end{equation}

\subsection{Interpretation of the parameters}

In this PDMP extension, the parameters can be interpreted as follows:

\begin{itemize}
  \item \(c(t)\) : consumption of the addictive good at time \(t\).
  \item \(A(t)\) : accumulated addiction capital at time \(t\).
  \item \(\phi\) : sensitivity (elasticity) of instantaneous utility with respect to consumption \(c\). A larger \(\phi\) means utility is more sensitive to changes in \(c\).
  \item \(\omega\) : sensitivity (elasticity) of instantaneous utility with respect to addiction \(A\). A larger \(\omega\) increases the impact of addiction capital on preferences (tolerance, reinforcement).
  \item \(\delta\) : rate of decline (depreciation) of addiction capital over time; it captures natural recovery/forgetting in the absence of consumption.
  \item \(\rho\) : intertemporal impatience (preference for the present); a higher \(\rho\) means the individual discounts future utility more heavily.
  \item \(\theta_1\) : absolute jump size in addiction capital; each jump models a relapse episode that suddenly increases \(A(t)\) (binge, crisis, acute stress).
  \item \(\theta_2\) : relative (or multiplicative) jump size in addiction capital;
  \item \(N(t)\) : counting process that records the number of relapse events up to time \(t\).
  \item \(\lambda(a)\) : jump (relapse) intensity when the addiction capital is \(A(t)=a\).
  \item \(\lambda_1\) : baseline relapse intensity, i.e. the risk of a relapse episode even at very low addiction levels (background vulnerability).
  \item \(\lambda_2\) : effect of current addiction on relapse intensity; a larger \(\lambda_1\) means that higher addiction capital makes relapses more likely, creating a feedback loop between \(A(t)\) and the occurrence of relapse jumps.
\end{itemize}
\subsection{Existence and monotonicity}
The control processes $c$ taking their values in $\mathbb{R}_+$ are considered \emph{admissible} if they are predictable and locally in time square integrable. Then, \cite{Graham1992} (e.g. Theorem 1.2) guarantee the existence and uniqueness of the solution.\\
Furthermore, we have the consistency result.
\begin{prop}\label{PropPoz}
    Let us assume that $\delta>0, \theta_1>0$, and $\theta_2>-1$. Then, if $c$ is a $\mathbb{R}_+^*$-valued predictable process and $A(0)\geq 0$, then $A\geq 0$, $\mathcal{L}eb\otimes\mathbb{P}$-a.s.\footnote{Here and after, $\mathbb{R}_+^*$ stands for strictly positive reals, $\mathbb{R}_+^*$ for non-negative reals, and $\mathcal{L}eb$ for Lebesgue measure on the real axis; "a.s." reads almost surely.}
\end{prop}
The second useful tool is a comparison result for the solutions.

\begin{prop}\label{PropComp}
    Let us assume that $\delta>0, \theta_1>0$, and $\theta_2\geq 0$. If $c_1\leq c_2$ are $\mathbb{R}_+^*$-valued admissible controls, and $A_1(0)\leq A_2(0)$, then the associated trajectories $A_j$ (starting at $A_j(0)$ and controlled with $c_j$) satisfy\[A_1\leq A_2,\ \mathcal{L}eb\otimes\mathbb{P}-\text{a.s. on }\mathbb{R}_+\times \Omega.\] 
\end{prop}

\medskip 

The proofs are postponed to the Appendix.
\section{The Capped Framework : Hamilton-Jacobi-Bellman Approach}\label{Sec3}
In this framework, we assume \[0<A_{\max}<\infty,\]but we allow $\theta_2<0$ (recovery in the addiction level). In this framework, since addiction can only get to $0$, we further need to assume
\begin{equation}\label{Asst2}
    \theta_2\geq -1.
\end{equation}
We further assume the control (instantaneous consumption rate) satisfies
\[
0 < c_{\min} \le c(t) \le c_{\max} < \infty.
\]
From a technical point of view, such compact control sets are standard in continuous-time optimal control, since they ensure the existence of optimal controls via Weierstrass-type arguments and existence theorems for infinite-horizon problems \cite{StokeyLucasPrescott1989,SeierstadSydsaeter1987}. In economic applications, bounded controls are routinely imposed to reflect physical and budget constraints on decision variables and to keep the Hamiltonian finite in the HJB formulation \cite{StokeyLucasPrescott1989,SeierstadSydsaeter1987}.

\subsection{Generator and HJB equation}

Let \(V(a)\) denote the value function when the current addiction capital is \(a\), that is the function defined by \eqref{Cost} where the maximum is taken over open-loop feedback controls $c(\cdot)$ taking their values in $\pp{c_{\min},c_{\max}}$. The infinitesimal generator \(\mathcal{L}^c\) of the controlled PDMP \eqref{eq:PDMP-classic-A} applied to a smooth test function \(f\) is
\begin{equation}
  \mathcal{L}^c f(a)
  = f'(a)\bigl(c - \delta a\bigr)
    + \lambda(a)\bigl[f(a+J(a)) - f(a)\bigr].
\end{equation}

Formally, the Hamilton--Jacobi--Bellman equation is
\begin{equation}\label{HJB}
\begin{split}
  &\rho V(a)
  - H(a,V(\cdot),V'(a))=0,\text{ where}\\
  &H(a,V(\cdot),V'(a)):=\\
  &\qquad\min_{c\in\pp{c_{\min},c_{\max}}}\pr{\mathcal{L}^c V(a)-\min\set{a,A_{\max}}^{\omega} c^{-\phi}}.
  \end{split}
\end{equation}
This formal equation can be made rigorous as follows.
\begin{theorem}\label{Th1}
    The value function $V$ is the unique bounded, uniformly continuous viscosity solution to \eqref{HJB} (in the sense of \cite[Definition on page 1112]{Soner1986} with $K=\mathbb{R}$).
\end{theorem}
\medskip

The proof follows from a rather direct application of \cite[Theorem 1.1]{Soner1986}. We only need to check the appropriate assumptions, and this is relegated to Section \ref{SecApp}.
\subsection{Optimality considerations}
\begin{theorem}\label{Th2}
    The problem \eqref{Cost} admits an optimal control. If the (optimal) value function is differentiable almost everywhere, then the optimal control is of bang-bang type, i.e., 
    \begin{equation}
        \label{CtrlOpt}
        c^{opt}=\begin{cases}
            c_{\min},\ \text{ if }\pr{V^{opt}}'(a)>- \phi \min\set{a,A_{\max}}^\omega c_{\max}^{-\phi-1},\\
            c_{\max},\ \text{ if }\pr{V^{opt}}'(a)<- \phi \min\set{a,A_{\max}}^\omega c_{\min}^{-\phi-1}.
        \end{cases}
    \end{equation}
    In particular, switches can only happen at points of discontinuity for $\pr{V^{opt}}'$.
\end{theorem}
\begin{proof}
\underline{Existence of optimal policies.} We begin with noting that the set
\begin{align*}
    \mathcal{I}_0(a):=\Big\{(&r,c-\delta a,(\lambda_0+\lambda_1a)(\theta_1+\theta_2\min\set{a,A_{\max}}\ :\\&r\geq \min\set{a,A_{\max}}^\omega c^{-\phi},\ c\in[c_{\min},c_{\max}]\Big\}
\end{align*}
is convex for every $a\in\mathbb{R}_+$. This is a simple consequence of the linearity of the second term in the parenthesis and the fact that $c\mapsto \min\set{a,A_{\max}}^\omega c^{-\phi}$ is convex and the $r$ component describes the epigraph of a convex function.\\
As a consequence, the optimality results in \cite[Theorem 1]{Vermes1986} apply to get relaxed optimality and the comments on \cite[Page 203]{Vermes1986} (also on \cite[Page 168]{Vermes1986}) on \emph{simple strategies} provide the existence of usual optimal controls.\\
\underline{Bang-bang result.} Let us note that, for any $p\in\mathbb{R}$ fixed, the function
\[f_0(c):=pc-\min\set{a,A_{\max}}^\omega c^{-\phi},\ c\in\pp{c_{\min},c_{\max}},\]
satisfies $f_0''(c)=-a^\omega \phi(\phi+1)c^{-\phi-2}<0$ which implies that any interior critical point of $f_0$ can only provide a maximum, thus ruling out the existence of singular arcs. As such, the optimal control, should it exist, is of bang-bang form.\\
If $a\in \mathbb{R}$ is a derivability point for $V^{opt}$, and $\pr{V^{opt}}'(a)+\phi \min\set{a,A_{\min}}^\omega c_{\max}^{-\phi-1}> \phi \min\set{a,A_{\min}}^\omega c^{-\phi-1},\ \forall c\in\pp{c_{\min},c_{\max}}$, then the switching function corresponding to $f_0$ with $p=\pr{V^{opt}}'(a)$ is non-decreasing, and the optimal control minimizing the Hamiltonian is $u_{\min}$. The remaining optimality of $u_{\max}$ is similar.
\end{proof}

\section{The Uncapped Framework $A_{\max}=+\infty$}\label{Sec4}
\subsection{Dynamics and HJB Equation}
In this case, we deal with linear dynamics 
\begin{equation*}
  dA(t) = \bigl[c(t) - \delta A(t)\bigr]\,dt + \pr{\theta_1+\theta_2A(t^-)}\,dN(t).
\end{equation*}
Alternatively, this can be written with respect to an exogenous Poisson random measure $M(dt,du)$ whose compensator is \[\hat M(dt,du)=\mathbf{1}_{\mathbb{R}_+^*}(u)dudt,\]by writing
\begin{align*}
  dA(t) = &\bigl[c(t) - \delta A(t)\bigr]\,dt \\+ &\pr{\theta_1+\theta_2A(t^-)}\int_{\mathbb{R}_+^*}\mathbf{1}_{u\leq \lambda(A(t^-))}\,M(dt,du).
\end{align*}
Equivalently, one generates a family of i.i.d. $[0,1]$ uniformly-valued random variables $U_k$. The first jump happens at 
\[
\tau_1 := \inf\Bigl\{t\geq 0:\ \int_0^t \lambda\bigl(A(s)\bigr)\,ds \geq -\ln U_1\Bigr\},
\]
where the solution can be explicitly computed through a variation-of-constants approach as 
\begin{align*}
    A(t)=e^{-\delta t}\pr{A(0)+\int_0^te^{\delta s}c(s)\,ds},\ \tau_1\ge t\ge 0.
\end{align*}
In this case, the infinitesimal generator is
\begin{equation}
  \mathcal{L}^c f(a)
  = f'(a)\bigl(c - \delta a\bigr)
    + \lambda(a)\bigl[f(a+J(a)) - f(a)\bigr],
\end{equation}
with \begin{equation}
    \begin{cases}
        \lambda(a)=&\lambda_1+\lambda_2 a,\\
        J(a)=&\theta_1+\theta_2 a,\ a\in\mathbb{R}_+.
    \end{cases}
\end{equation}
Formally, the Hamilton--Jacobi--Bellman equation is
\begin{equation}\label{HJB2}
\begin{split}
  &\rho V(a)
  - H(a,V(\cdot),V'(a))=0,\text{ where}\\
  &H(a,V(\cdot),V'(a)):=\inf_{c\in\mathbb{R}_+^*}\pr{\mathcal{L}^c V(a)-a^{\omega} c^{-\phi}}.
  \end{split}
\end{equation}
The rigorous arguments can be obtained by passing to the limit $c_{\min}\rightarrow 0^+,\ c_{\max}\rightarrow\infty$, and $A_{\max}\rightarrow\infty$. Since the Hamiltonian needs to be finite, the infimum can only happen at a critical point, i.e.
\[V'(a)=-\phi a^\omega c^{-1-\phi},\] which, provided that $V$ is found to be absolutely continuous, identifies $c$ as a feedback control
\[c(a)=\pr{\frac{-\phi a^\omega}{V'(a)}}^{\frac{1}{1+\phi}}.\] By recalling that $V(\alpha)-V(\beta)=\int_\beta^\alpha V'(r)\, dr$, and with the assumption that $V(0)=0$, this leads to 
\begin{align*}
    &\rho\phi\int_0^a r^\omega c^{-1-\phi}(r)\,dr-\phi a^\omega c(a)^{-1-\phi}(c(a)-\delta a)\\&-a^\omega c(a)^{-\phi}-\phi\pr{\lambda_1+\lambda_2 a}\int_a^{(1+\theta_2)a+\theta_1} r^\omega c^{-1-\phi}(r)\,dr=0,
\end{align*}or, again, with $\zeta:=c^{-1-\phi}$,
\begin{equation}\label{VH} \begin{split}
0
&= \rho\phi \int_0^a r^\omega \zeta(r)\,dr
  - \phi a^\omega \zeta(a)\bigl(\zeta(a)^{-\tfrac{1}{1+\phi}} - \delta a\bigr)
  \\&- a^\omega \zeta(a)^{\tfrac{\phi}{1+\phi}}
  - \phi\bigl(\lambda_1 + \lambda_2 a\bigr)
    \int_a^{(1+\theta_2)a + \theta_1} r^\omega \zeta(r)\,dr.
\end{split}
\end{equation}
\subsection{The explicit solution using power ansatz}
For general parameters
\[
\lambda_1,\lambda_2,\theta_1,\theta_2,\delta,\omega,\phi, \rho
\]
and without imposing a specific parametric form on $\zeta$, the equation \eqref{VH} belongs to the nonlinear mixed Volterra-Fredholm integral equation of the second kind which does not admit explicit solutions. The approach needs to be numerical.\\
Assume a general power law
\[
\zeta(r)=K r^{\alpha},\qquad K\neq 0,\ \alpha\in\mathbb{R}.
\]
For the equation to hold for all \(a>0\) with a single power of \(a\), exponents must match which forces
\[
\alpha=-1-\phi,
\]
and this can only be obtained when $\lambda_2=0$ and either $\lambda_1=0$ or $\theta_1=0$.\\
With \(\zeta(r)=K r^{-1-\phi}\), all integrals are elementary and every term scales like \(a^{\omega-\phi}\) in the feasible cases (see hereafter), so the equation reduces to an algebraic condition on \(K\). We assume \[\omega>\phi+1.\]

\paragraph{No jumps (pure drift)}

Assume
\begin{equation}\label{Case1}
    \lambda_1=\lambda_2=0.
\end{equation}

The equation simplifies to
\begin{align*}
0
&= \rho\phi \int_0^a r^\omega \zeta(r)\,dr
  - \phi a^\omega \zeta(a)\bigl(\zeta(a)^{-\tfrac{1}{1+\phi}} - \delta a\bigr)
  \\&- a^\omega \zeta(a)^{\tfrac{\phi}{1+\phi}}.\end{align*}
Substituting \(\zeta(r)=K r^{-1-\phi}\) gives
\begin{align*}
0
&= \left(
    \frac{\rho\phi K}{\omega-\phi}
    + \phi\delta K
    - \phi K^{\tfrac{\phi}{1+\phi}}
    - K^{\tfrac{\phi}{1+\phi}}
   \right) a^{\omega-\phi},
\end{align*}

hence,
\(
K
=\left(
\frac{\phi+1}{\pr{\dfrac{\rho}{\omega-\phi}+\delta}\phi}
\right)^{1+\phi}.
\)

\paragraph{Constant intensity and proportional jumps}
Assume
\begin{equation}\label{Case2}
\lambda_2=0,\quad \theta_1=0,\quad \theta_2>-1.
\end{equation}
In this setting,
\[
\int_a^{(1+\theta_2)a} r^{\omega}\zeta(r)\,dr
=\frac{K}{\omega-\phi}\Bigl((1+\theta_2)^{\omega-\phi}-1\Bigr)a^{\omega-\phi}.
\]
The equation \eqref{VH} (up to a multiplicative term $a^{\omega-\phi}$) becomes
\begin{equation*}
\begin{split}
0
&=\frac{\rho\phi K}{\omega-\phi}-(\phi+1)K^{\frac{\phi}{1+\phi}}+\delta \phi K
\\&\quad -\lambda_1\phi K\frac{(1+\theta_2)^{\omega-\phi}-1}{\omega-\phi}.
\end{split}
\end{equation*}
This leads to
$
K^*
=\left(
\frac{\phi+1}{\left(
\frac{\rho-\lambda_1\bigl((1+\theta_2)^{\omega-\phi}-1\bigr)}{\omega-\phi}
+\delta
\right)\phi}
\right)^{1+\phi}.
$

We have proven the following \begin{theorem}[Explicit value functions for \(\lambda_2=0\)]\label{Th3}
  Assume
  $
    \lambda_2=0,\quad \theta_2>-1,\quad \text{and }\lambda_1\theta_1=0.
  $
  Furthermore, assume that \(\omega>\phi+1.\)
  \begin{enumerate}
    \item The optimal control is linear in \(a\):
    \begin{equation}\label{copt}
      \begin{cases}
        c^*(a) &= \kappa^* a,\\[0.5em]
        \displaystyle \kappa^*
        &= \frac{\phi}{\phi+1}\left(
\frac{\rho-\lambda_1\bigl((1+\theta_2)^{\omega-\phi}-1\bigr)}{\omega-\phi}
+\delta
\right).
      \end{cases}
    \end{equation}
    \item The optimal value function is
    \begin{equation}
      V^*(a)
      = 
      - \frac{K^*\,\phi}{\omega-\phi}\,a^{\omega-\phi}.
    \end{equation}
  \end{enumerate}
\end{theorem}
\begin{figure}[h]
  \centering
  \begin{subfigure}[t]{0.48\columnwidth}
    \centering
    \includegraphics[width=\linewidth]{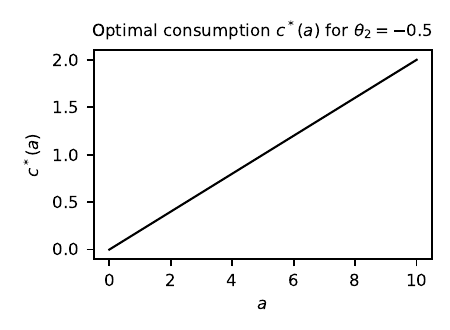}
    \caption{}
    \label{fig:cstar-theta2minus05}
  \end{subfigure}
  \hfill
  \begin{subfigure}[t]{0.48\columnwidth}
    \centering
    \includegraphics[width=\linewidth]{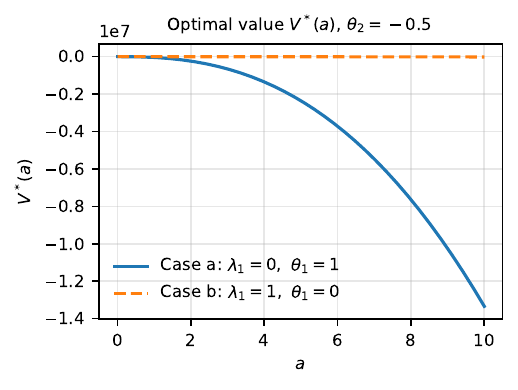}
    \caption{}
    \label{fig:Vstar-theta2minus05}
  \end{subfigure}

  \caption{Optimal consumption $c^*(a)$ and value $V^*(a)$
    for $\theta_2 = -0.5$ ($\phi   = 2.0,\ \omega = 4.5,\ \rho   = 0.05,\ \delta = 0.02,\ \lambda_2=0$).}
  \label{fig:theta2minus05-onecolumn}
\end{figure}
\section{Implications for Addiction}\label{Sec5}

The optimality theorems have several implications for the addiction dynamics.

\subsection{Implications for addiction in the uncapped framework.}
\textbf{- Proportional dependence of consumption on addiction capital.}
  The optimal feedback control is
  \[
    c^*(a) = \kappa^* a,
  \]
  so that, as the addiction capital $a$ increases, the optimal current consumption $c^*(a)$ also increases proportionally. This captures a \emph{reinforcing} effect: when the addiction stock is high, the marginal effective dose needed for utility is higher in absolute terms, and it becomes optimal to maintain (or raise) current consumption in line with $A(t)$. This differs from the inverse dependence in the original Becker--Murphy model and arises here from the particular HJB structure and sign convention.

\textbf{- Low consumption at low addiction capital.}
  For small $a>0$, the optimal policy prescribes relatively low consumption levels $c^*(a)$, corresponding to mild consumption episodes at early stages of addiction (or immediately after quitting) when the addiction capital is low.

\textbf{- Effect of relapse or intervention shocks.}
  The Poisson parameters $\lambda_1$ and $\theta_2$ describe jump risk. If $\theta_2>0$ (relapses or positive shocks that increase $A$), then larger $\lambda_1$ increases the term $(1+\theta_2)^{\omega-\phi}-1$ and thus lowers
  \[
    \frac{\rho - \lambda_1\bigl((1+\theta_2)^{\omega-\phi}-1\bigr)}{\omega-\phi},
  \]
  which propagates to $K^*$ and $\kappa^*$; depending on parameter values, this may either strengthen or attenuate the proportionality of $c^*(a)$ to $a$. If $\theta_2\in(-1,0)$ (negative shocks such as partial resets due to treatment), the sign of this contribution reverses, tending to weaken the reinforcement and reduce optimal consumption for given $a$. Mathematically, the value function is finite only if the averaged multiplicative effect is mitigated by low average number of jumps per time unit $\lambda_1<<1$, or by the large preference parameter $\rho>>1$.

\textbf{- Long--run behavior.}
  Under the stationary feedback $c^*(a)=\kappa^* a$, the addiction capital $A(t)$ evolves toward a stochastic steady regime in which the linear drift $\kappa^* A(t)-\delta A(t)$ and the jump component balance, provided $\kappa^*>0$ and suitable integrability conditions hold. In this regime, the agent optimally maintains a nonzero addiction capital and adjusts consumption proportionally to $A(t)$ and to the parameters governing depreciation, discounting, and jump risk.
\subsection{Implications for addiction in the capped framework.}
\textbf{- Alternating addiction phases.}
From the point of view of addiction dynamics, the bang--bang control rule in Theorem \ref{Th1} implies that the addict never chooses an interior level of consumption: at each addiction level $a$, it is optimal either to consume at the lowest feasible rate $c_{\min}$ (a ``recovery/maintenance phase'') or at the highest feasible rate $c_{\max}$ (a ``binge phase''). The choice between these two regimes is governed by the sign and magnitude of $(V^{opt})'(a)$, which measures the marginal value (or marginal cost) of additional addiction capital.

Regime switches occur at threshold addiction levels where the marginal value of addiction capital changes abruptly.

\section{Conclusions and Perspectives}\label{Sec6}
This classical PDMP formulation keeps the same economic structure as the original Becker-Murphy model but adds shocks by occasional, state-dependent relapse jumps, while preserving a clear interpretation of all parameters in terms of addiction dynamics and relapse risk.\\
The uncapped framework both connects our analysis to existing results in the literature and yields explicit feedback controls together with closed-form value functions.\\
The optimality analysis shows that the capped framework, in which consumption is bounded, is the natural representation of environments where the availability of addictive goods is limited by regulatory, market, or physiological constraints. From this perspective, the bang-bang structure of optimal controls and the existence of an addiction cap make the capped model a more realistic description of actual addictive behavior than the purely uncapped benchmark. At the same time, our formulation treats the cap via a hard state constraint; a natural extension is to replace this by genuinely reflected dynamics at endogenous addiction boundaries, using reflected PDMP techniques.\\

Future work will also include numerical schemes for the associated HJB equations with jumps and the calibration of model parameters to empirical data on consumption paths and relapse episodes, allowing for quantitative evaluation of policy interventions in addiction markets.

\section{Appendix}\label{SecApp}
\begin{proof}\textbf{(Proof or Proposition \ref{PropPoz})}
    One notes that $c-\delta a\mid_{a=0}\geq 0$, such that the deterministic part keeps $A$ non-negative if the starting point is non-negative. Second, at the first jump time $\tau_1$, since $\theta_2\geq -1$, and $\theta_1\geq 0$, it follows that $A(\tau_1)\geq 0$. The argument can then be repeated on every $\pp{\tau_k,\tau_{k+1}}$, for $k\geq 1$.
\end{proof}
\begin{proof}\textbf{(Proof of Proposition \ref{PropComp})} 
    The first remark concerns the monotonicity of the deterministic dynamics. If $c_1\leq c_2$ are two Borel-measurable $\mathbb{R}_+$-valued control policies, then $\tilde A_1\leq \tilde A_2$, where 
\[d\tilde A_j=\pp{c_j(t)-\delta \tilde A_j(t)}\, dt,\ j\in\set{1,2}.\]
This is standard, but it also follows from the explicit solution of the linear equation satisfied by $\tilde A_1-\tilde A_2$.\\
Second, the jump mechanism can be simulated using a Poisson measure $M$ on the extended space $\mathbb{R}_+$ whose compensator is \(\hat{M}(ds,du)=dsdu,\) by considering
\[d\tilde{A}_t=(c(t)-\delta\tilde A(t))dt+J(\tilde A(t^-))\mathbf{1}_{u\leq \lambda(\tilde A(t^-))}M(dt,du).\]
Note that, due to monotonicity of $\lambda$,  \[\set{u\leq \lambda(\tilde A_1(t^-))}\subset \set{u\leq \lambda(\tilde A_2(t^-))},\] hence, by invoking the monotonicity of $J$,
\[J(\tilde A_1(t^-))\mathbf{1}_{u\leq \lambda(\tilde A_1(t^-))}\leq J(\tilde A_2(t^-))\mathbf{1}_{u\leq \lambda(\tilde A_2(t^-))}.\]This implies that the post-jump position maintains the order, and our result is complete by recurrence over the jumping times (which are commonly generated by $M$).
\end{proof}
\begin{proof}\textbf{(Proof of Theorem \ref{Th1})}
    The proof follows from a direct application of \cite[Theorem 1.1]{Soner1986}. Indeed, $a\mapsto b(a):=c-\delta a$ is Lipschitz-continuous, thus taking care of assumption (1.1) in \cite{Soner1986}. The jump intensity $\lambda\geq 0$, taking care of the assumption (1.5) in \cite{Soner1986}, and both $J$ and $\lambda$ are bounded.\\
    Although $b$ is not bounded, one easily notes that the deterministic part $\frac{dA}{dt}=c-\delta A$ keeps $\mathbb{R}_+$ invariant, and jumps only increase the addiction capital, thus keeping $\mathbb{R}_+$ invariant (see also \cite{GoreacViab2012}). As such, assumption (1.2) in \cite{Soner1986} can also be dealt with. Finally, the weak continuity of the post-jump measure (i.e., the continuity of
    $a\mapsto h(a+J(a))$, for continuous functions $h$) takes care of assumption (1.3) in \cite{Soner1986}. The assumption (1.4) in \cite{Soner1986} is irrelevant for non-constrained dynamics.
\end{proof}
\section*{Aknowledgements}
D.G. acknowledges financial support
from National Sciences and Engineering Research Council (NSERC), Canada, Grant/Award RGPIN-
2025-03963. D. G. and J. L. acknowledge financial support from the NSF of Shandong Province
(ZR2023ZD35), the NSF of People's Republic of China (W2511002, 12031009), and the National
Key R and D Program of China (2018YFA0703900).


\begin{thebibliography}{99}
\bibitem{BeckerMurphy1988}
G.~S. Becker and K.~M. Murphy,
``A theory of rational addiction,''
\emph{J. Polit. Econ.}, vol.~96, no.~4, pp.~675--700, 1988.

\bibitem{Campillo2019}
F.~Campillo,
``Introduction to (piecewise deterministic) Markov processes and applications in biology,''
Lecture notes, BCAM course, 2019. [Online]. Available: \url{http://www-sop.inria.fr/members/Fabien.Campillo/wp-content/uploads/pdf/slides-2019-bcam.pdf}

\bibitem{Davis1984}
M.~H.~A. Davis, 
``Piecewise-deterministic Markov processes: A general class of non-diffusion stochastic models,'' 
\emph{J. Roy. Statist. Soc. Ser. B}, vol.~46, no.~3, pp.~353--388, 1984.

\bibitem{Davis1993}
M.~H.~A. Davis, 
\emph{Markov Models and Optimization}. 
London, U.K.: Chapman \& Hall, 1993.

\bibitem{CostaDufour2008}
O.~L.~V. Costa and F.~Dufour, 
``The vanishing discount approach for the average continuous control of piecewise deterministic Markov processes,'' 
\emph{J. Appl. Probab.}, vol.~45, no.~3, pp.~742--756, 2008.

\bibitem{Dufour2016}
F.~Dufour and M.~H. Costa, 
``Optimal impulsive control of piecewise deterministic Markov processes,'' 
\emph{Stochastics}, vol.~88, no.~1, pp.~85--104, 2016.

\bibitem{GoreacViab2012}
D.~Goreac, 
``Viability, invariance and reachability for controlled piecewise deterministic Markov processes associated to gene networks,'' 
\emph{ESAIM Control Optim. Calc. Var.}, vol.~18, no.~2, pp.~401--426, 2012.

\bibitem{GoreacSICON2015}
D.~Goreac, 
``Asymptotic control for a class of piecewise deterministic Markov processes associated to temperate viruses,'' 
\emph{SIAM J. Control Optim.}, vol.~53, no.~4, pp.~1860--1891, 2015.

\bibitem{GoreacSereaZubov2012}
D.~Goreac and O.-S. Serea,
``Linearization techniques for controlled piecewise deterministic Markov processes; application to Zubov's method,''
\emph{Appl. Math. Optim.}, vol.~66, no.~1, pp.~27--48, 2012.

\bibitem{Graham1992}
C.~Graham, ``McKean--Vlasov It\^o--Skorohod equations, and nonlinear
  diffusions with discrete jump sets,'' \emph{Stochastic Processes and
  Their Applications}, vol.~40, no.~1, pp.~69--82, 1992.

\bibitem{GruberKoszegi2001}
J.~Gruber and B.~K{\H o}szegi,
``Is addiction ``rational''? Theory and evidence,''
\emph{Q. J. Econ.}, vol.~116, no.~4, pp.~1261--1303, 2001.

\bibitem{SeierstadSydsaeter1987}
A.~Seierstad and K.~Syds{\ae}ter, 
\emph{Optimal Control Theory with Economic Applications}. 
Amsterdam, The Netherlands: North-Holland, 1987.

\bibitem{Soner1986}
H.~M. Soner, 
``Optimal control with state-space constraint. II,'' 
\emph{SIAM J. Control Optim.}, vol.~24, no.~6, pp.~1110--1122, 1986.

\bibitem{StokeyLucasPrescott1989}
N.~L. Stokey, R.~E. Lucas, Jr., and E.~C. Prescott, 
\emph{Recursive Methods in Economic Dynamics}. 
Cambridge, MA, USA: Harvard Univ. Press, 1989.

\bibitem{AddictionNeuroModel2022}
R.~Verdejo-Garc{\'i}a \emph{et al}.,
``Computational models of behavioral addictions: State of the art and future directions,''
\emph{Curr. Opin. Behav. Sci.}, vol.~45, pp.~101--110, 2022.

\bibitem{Vermes1986}
D.~Vermes, 
``On the optimal control of piecewise-deterministic Markov processes,'' 
\emph{Stochastics}, vol.~17, no.~3, pp.~165--188, 1986.

\bibitem{YangZhang2022}
Z.~Yang and X.~Zhang,
``A stochastic model of rational addiction,''
\emph{Ann. Econ. Finance}, vol.~23, no.~2, pp.~223--251, 2022.
\end{thebibliography}
\end{document}